\documentclass[12pt]{amsart}
\usepackage[utf8]{inputenc}
\usepackage[T1]{fontenc}
\usepackage[top=3cm, bottom=3cm, left=3cm, right=3cm]{geometry}
\usepackage{csquotes, amsfonts, amsmath, amsthm, amssymb, stmaryrd, dsfont, enumitem, xcolor, url, graphicx, multicol, float, setspace, hyperref, tikz, multirow, esvect, colortbl, pifont, doi, url}
\usepackage[normalem]{ulem}
\usepackage[breakable, skins]{tcolorbox}
\usepackage{mathrsfs}

\setlist[itemize]{leftmargin=1cm}
\setlist[enumerate]{leftmargin=1cm}

\theoremstyle{plain}
\newtheorem{theo}{Theorem}[section]

\theoremstyle{definition}
\newtheorem{defi}[theo]{Definition}

\newtheorem{expl}[theo]{Example}

\hypersetup{
    colorlinks=true,
    linkcolor=black,
    urlcolor=black,
    citecolor=black
}
\newcommand\GF{{\mathcal F}}
\newcommand\GG{\mathcal G}
\newcommand\GH{{\mathcal H}}
\newcommand\Id{{\rm Id}}
\newcommand\LH{{\mathcal L}({\mathcal H})}
\newcommand\BX{{\mathcal L}_1(X)}
\newcommand\LX{{\mathcal L}({X})}
\newcommand{\tpos}{{t \geqslant 0}}
\newcommand{\CC}{\mathbb{C}}
\newcommand{\DD}{\mathbb{D}}

\newcommand{\RR}{\mathbb{R}}
\newcommand{\TT}{\mathbb{T}}

\def\ol{\overline}
\def\ran{\mathop{\rm ran}\nolimits}

\renewcommand{\textbf}[1]{\begingroup\bfseries\mathversion{bold}#1\endgroup}

\title[The Invariant Subspace Problem]
{Recent perspectives on the Invariant Subspace Problem}

\author{I. Chalendar}
\address{Isabelle CHALENDAR, Université Gustave Eiffel, LAMA, (UMR 8050), 
    UPEM, UPEC, CNRS, F-77454, Marne-la-Vallée (France)}
\email{isabelle.chalendar@univ-eiffel.fr}

\author{J. R. Partington}
\address{Jonathan R. PARTINGTON, School of Mathematics, University of Leeds, Leeds LS2 9JT, Yorkshire, U.K.}
\email{j.r.partington@leeds.ac.uk}


\begin{document}

\begin{abstract}
We review recent work connected with the invariant subspace problem for operators, in particular new developments in
the last 15 years. In particular, we include discussions of almost-invariant subspaces,
universal operators, specific classes of operators and new results in the framework of Banach spaces.

\end{abstract}

\maketitle
  
\section{Introduction}

The invariant subspace problem, which dates back at least to the time of von Neumann, over 75 years ago,
asks whether every bounded operator $T$ on a complex Banach space $X$ of dimension at least two has a nontrivial closed invariant subspace: that is, a closed subspace $M$ of $X$  such that $M\neq \{0\}$, $M\neq X$ and $T(M) \subseteq M$.
The invariant subspace problem is a cornerstone of operator theory, as the existence of invariant subspaces provides insight into the structure and behaviour of operators, with applications ranging from spectral theory to dynamical systems. Another related problem even more interesting and certainly more diffcult is the complete  description of the lattice of all the invariant subspaces of a given bounded operator $T$. There are an impressive number of sufficient conditions that guarantee the existence of invariant subspaces, using a wide variety of tools that contribute to attractiveness of this problem. However there are very few operators for which the lattice of invariant subspaces is perfectly described. In other words, even if an answer to the invariant subspace problem  is provided, it will not be the end of the interest in invariant subspaces, there will still be much to do!     \\

In \cite{CP13} we included a detailed history of the problem, which we shall not repeat,
although we mention also   the classic text \cite{RR03} and
the more recent book \cite{CP11} as useful references on the subject.

We note that the question was settled in the negative for Banach spaces $X$ with counterexamples
from Enflo~\cite{enflo} and Read~\cite{read84}. 
While the problem remains open for Hilbert spaces, partial results have been obtained for specific operator classes, such as compact operators and subnormal operators, which we revisit in light of recent developments.
\\

In \cite{CP13} we  examined the following topics in particular detail:

\textbullet\ universal operators;

\textbullet\ Bishop operators; and

\textbullet\ (finitely) strictly singular operators.\\

Here we shall focus on the following  topics:

\textbullet\ Section \ref{sec:2}: almost-invariant subspaces (a notion developed very recently);

\textbullet\ Section \ref{sec:3}:
universal operators again (there has been further research in this area);

\textbullet\ Section \ref{sec:4}: particular operator classes; and finally 

\textbullet\ Section \ref{sec:5}:
 the situation for Banach space operators, finishing with the
notion of  ``typical'' operators and their properties.

These topics are interconnected, offering a comprehensive view of recent progress in tackling the invariant subspace problem.
This paper in no way claims to make an exhaustive list of important recent contributions to this exciting open issue. Rather, it is  a question of showing the abundant research activity around, by detailing some remarkable directions obtained by researchers from a large number of countries.    
\\

{\em Some notation.} We write $\LX$ for the algebra of all bounded operators on
a Banach space $X$. The identity operator is denoted by $\Id$.

\section{Almost-invariant subspaces}
\label{sec:2}

Although the invariant subspace problem for Hilbert spaces does not seem to be
solved at the time of writing, there is a sense in which it is ``almost'' solved.
In 2009 Androulakis, Popov, Tcaciuc and Troitsky~\cite{APTT09} initiated the study of almost-invariant subspaces,
which we now define.

\begin{defi}\label{def:almostis}
A subspace $Y$ of a Banach space $X$ is said to be an 
{\em almost-invariant
subspace\/} for 
an operator $T$ on $X$ if there exists a finite-dimensional subspace $F$ such that $T Y \subseteq Y + F$. The
minimal dimension of such a subspace $F$ is called the 
{\em defect\/} of $Y$ for $T$. 
\end{defi}

Clearly, if $Y$ has finite dimension or codimension it is automatically
almost invariant, so our interest here lies in spaces that have infinite
dimension and codimension in $X$. These will be known as {\em half spaces}.

It turns out that the ``almost-invariant subspace problem'' is solved,
at least for Hilbert spaces, since 
 Popov and Tcaciuc \cite{PT13} proved the following theorem. Some  
 partial results 
were given earlier by
Androulakis,  Popov,  Tcaciuc and 
              Troitsky, while
              alternative proofs and 
              related results were given later by
Jung, Ko and Pearcy    \cite{JKP18}.
There are related results due to Sirotkin and Wallis \cite{SW14,SW16}.

 \begin{theo}
 Let $T$ be a bounded operator on an infinite-dimensional reflexive Banach space $X$. Then $X$
 admits an almost-invariant half space with defect 1.
 \end{theo}
 
To give some flavour of the arguments used, we present the following weaker form of this theorem, and refer the interested reader to the
 original sources for   proofs of the stronger form.
 
 \begin{theo}
  Let $T$ be a bounded operator on a  Hilbert space $H$. Then 
  either $T$ has an eigenvalue (and hence an invariant subspace) or $T$
has an almost-invariant half space with defect 1.
 \end{theo}
 
 \begin{proof}
 Clearly, we may suppose without loss of generality that $H$ is infinite-dimensional.
 
 If $T$ has no eigenvalues, then there is a point $\lambda \in \partial\sigma(T)$
 that is not an eigenvalue.
 Note that if $T$ has an almost-invariant half space $Y$ of defect 1, that is,
 if $T(Y) \subseteq Y+F$, then $(T-\mu \Id)(Y) \subseteq Y+F$ for every $\mu \in \CC$.
 Therefore, to prove the theorem, we may suppose without loss of generality
 that $\lambda=0$; i.e., $0 \in \partial\sigma(T)$ and $0$ is not an eigenvalue.
 
 We may therefore find a sequence $(\lambda_n) $ in $\CC \setminus \sigma(T)$,
 with $\lambda_n \to 0$. Now we have $\|(\lambda_n \Id-T)^{-1}\| \to \infty$.
 We write $R_n=(\lambda_n \Id-T)^{-1}$ and observe that, by the uniform
 boundedness principle there is an $e \in X$ such that $\|e\|=1$ and $\|R_n e\| \to \infty$.
 Let us write $x_n=R_n e/\|R_n e\|$ for each $n$.
 
Thus we have
\begin{equation}\label{eq:txn}
(T-\lambda_n \Id)x_n= -e/\|R_n e\| .
\end{equation}
 As a bounded sequence in a Hilbert space, $(x_n)$ has a weakly convergent subsequence,
 and so, by passing to a subsequence and relabelling we may suppose that $x_n \xrightarrow{w} z \in H$.
 Now, since $\lambda_n \to 0$ and $\|R_n e \| \to \infty$, it follows from \eqref{eq:txn}
 that
$ \| Tx_n \|\to 0$,
 and so 
 \[
 \langle Tz,y \rangle =\langle  z,T^*y\rangle = \lim \langle x_n , T^*y \rangle =0
 \] for all $y \in H$,
 which implies that $Tz=0$. But $0$ is not an eigenvalue and so $z=0$.
 
 That is, $ x_n \xrightarrow{w} 0$, and so, by a classical theorem of Kadets and Pelczynski \cite{AK16,KP65}, $(x_n)$
 contains a basic subsequence; by passing to this subsequence we may suppose
 without loss of generality
 that $(x_n)$ is itself  a basic sequence.
 
 We now have an almost-invariant half space of defect 1, namely the closed linear span
 of $\{x_{2n}: n \ge 1 \}$. The corresponding space $F$ can be taken to be the linear span of $e$.
\end{proof}
 
Let us give an important example.
Recall that the Hardy space $H^2$
is the space of holomorphic functions on the open unit disc $\DD$ in th copmplex plane $\CC$ whose Taylor coefficients $(a_n)_{n\geq 0}$ satisfy $\sum_{n\geq 0}|a_n|^2<\infty$. 

\begin{expl}
It follows from Beurling's classical theorem that
in the Hardy space $H^2$ the model spaces $K_\theta:=H^2 \ominus \theta H^2$ with $\theta$ an inner function are the
invariant subspaces for the backward shift operator $S^*$ defined by $S^*  f(z) = \dfrac{f(z)-f(0)}{z}$.
Writing $T_G$ for the Toeplitz operator $f \mapsto P_{H^2} Gf$ with $G \in L^\infty(\TT)$
and $P_{H^2}: L^2(\TT) \to H^2$ the orthogonal (Riesz) projection,
we have that $K_\theta = \ker T_{\ol\theta}$. Here $\TT$ denotes the unit circle in $\CC$. 

In general Toeplitz kernels are not invariant under $S^*$, but we can show directly that
they are almost invariant with defect 1. 
In the nontrivial case there is a function $f \in \ker T_G$ with $f(0)=1$.
Now for any $g \in \ker T_G$ we can write $g=\lambda f+g_0$ where
$\lambda=g(0)$ and $g_0 \in \ker T_G$ with $g_0(0)=0$.
So $S^* g = \lambda S^*f + S^* g_0$.
But $S^* g_0  \in \ker T_G$ since $G S^* g_0 = G \ol z g_0 = \ol z (G g_0) \in \ol{H^2_0}=L^2(\TT) \ominus H^2$.

Hence $S^* (\ker T_G) \subseteq \ker T_G + \CC S^*f$; this is near invariance with defect $1$.
\end{expl}

Almost invariance for the backward shift is of particular interest as it is closely
related to the concept of near invariance, which goes back to Hayashi~\cite{hayashi}, Hitt~\cite{hitt88} and Sarason~\cite{sarason}.
A subspace $M \subset H^2$ is {\em nearly invariant\/} if whenever $f \in M$ with $f(0)=0$
we have $S^* f \in M$.  That is, we can divide out zeros without leaving the space.
The paper \cite{CGP20} describes the almost-invariant subspaces for the backward shift as well as those
subspaces that are nearly invariant with a defect: an extension to   vector-valued Hardy spaces can be found in \cite{CDP20}.\\

By considering perturbations of an operator, it is possible to go from almost-invariant subspaces to invariant subspaces.
Many years ago Brown and Pearcy~\cite{BP71} showed that every Hilbert space operator has
a compact perturbation with an invariant half-space.
Tcaciuc \cite{T19,T22} has recently gone much further, showing the following result.

\begin{theo}
Let $T$ be a bounded operator on a Banach space. Then there exists a rank-one operator $F$ such that $T+F$ has an invariant half-space. In addition one can choose the  operator $F$ to be of arbitrarily small norm. 
\end{theo}

\section{Universal operators}
\label{sec:3}

\subsection{A model for every bounded operator}

An operator $U$ on a Banach space $X$ is said to be {\em universal\/} if it models every bounded
operator on $X$: more precisely, if for every operator $T$ on $X$ there is a 
closed subspace $X_0 \subseteq X$ with $U(X_0) \subseteq X_0$, an
 isomorphism $\phi$ of $X$ onto $X_0$ and a scalar
$\lambda \ne 0$ such that $\lambda T= \phi^{-1}U_{|X_0}\phi$. 
This notion goes back to Rota \cite{rota59,rota60}, who gave the example of the
backward shift of infinite multiplicity.

Caradus \cite{caradus} gave a simple criterion for an operator to be universal,
namely (as originally stated) that $U$ should have a right inverse and that its kernel
should contain a subspace isomorphic to $X$.  

This concept is most commonly applied in the context of Hilbert spaces, in which
case the criterion may be stated simply as follows. 

\begin{theo}
	Let $U\in \LH$ where $\GH$ is a complex Hilbert space. Then $U$ is universal if $\{x\in \GH:Ux=0\}$ is infinite dimensional and $U\GH=\GH$.
\end{theo}

The notion of universality has stimulated much research, since
if we can find a universal operator for which we have a full description of the
lattice of invariant subspaces, we can solve the invariant subspace problem by
determining whether there is an infinite-dimensional invariant subspace that does
not contain any smaller invariant subspaces except $\{0\}$.\\

Many examples of universal operators $U$  are detailed in  \cite{CP11,CP13}. The examples are related  to  (weighted)  shifts or (weighted) composition operators. Their universality is proved using Caradus's theorem. The necessary condition to be universal, namely  ``$\{ x\in\GH:Ux=0\}$ is infinite dimensional'' is usually quite easy to check. Most of the time, the difficult part consists in verifying the sufficient condition, that is the surjectivity of $U$.   

The link between the invariant subspace problem and composition operators on the Hardy space $H^2$ follows from  \cite{NRW87}. In this beautiful paper, E. Nordgren, P. Rosenthal and F. S. Wintrobe  proved the universality of $C_\varphi -\Id$ on $H^2$ where $C_\varphi (f)=f\circ \varphi$ and $\varphi$ is the hyperbolic automorphism defined by $\varphi(z)=\frac{z+1/2}{1+z/2} $.  Their proof is detailed in \cite{CP13} and an alternative proof is  given in \cite{CG17} involving groups of Toeplitz operators.


	In \cite{pozzi12}, using the theory of semi-Fredholm operators, E. Pozzi obtained the following generalization of Caradus's theorem: 
\begin{theo}\label{th:caradus}
	An operator $U\in\LH$, where $\GH$ is a complex Hilbert space, is universal if its kernel is of infinite dimension  and the codimension of its range is finite.  
\end{theo}
  Pozzi  applied this criterion to provide  many examples of universal operators defined as  translations of weighted composition operators on the Hilbert space $L^2(0,1)$  as well as the Sobolev space $W_0^1$. 

In \cite{cn22}, J.R. Carmo and S.W. Noor proved that the composition operator $C_\phi$ associated with the affine  symbol $\phi(z)=(z+1)/2$ on the Hardy space $H^2$ is such that $C_\phi-\Id$ is universal. In other words,   a complete description of the closed invariant subspaces of $C_\phi$ on $H^2$ would solve the famous invariant subspace problem on Hilbert space.   
More general results from this paper lead to a proof that, with $C$ the Ces\`aro operator on $H^2$,
there is an analytic function of $C^*$ that is universal \cite{GPR25}.

In \cite{CG16}, C. Cowen and E. Gallardo-Guti\'errez proved that there exists a universal operator that commutes with a quasi-nilpotent injective compact operator with dense range. More exactly, this universal operator is  the adjoint of some analytic Toeplitz operator on the   Hardy space $H^2$.  This provide the existence of invariant subspaces in some particular cases, even though a universal operator is a sort of ``model'' of any operator and every operator commuting with a compact operator has a non trivial invariant subspace.     

In \cite{CG22}, the same authors exhibit an analytic Toeplitz operator whose adjoint is universal and which commutes with a quasinilpotent, injective, compact operator with dense range; unlike other examples, it acts on the Bergman space instead of the Hardy space and this operator is associated with a “hyperbolic” composition operator.

	In \cite{ST18}, R. Schroderus and H.O. Tylli \cite{ST18} looked more closely at some fundamental properties of the class of universal operators. In particular thy derived spectral theoretic consequences of universality which can be used to verify that a given operator is not universal. 
More precisely they proved the following result. 
\begin{theo}
For any universal operator $U\in\LH$, there exists $r>0$ such that the kernel of $U-\lambda \Id$
is of infinite dimension whenever $|\lambda|<r$. That is, the multiplicity of each eigenvalue of   $U$ in the disc $\{\lambda \in \CC: |\lambda|<r \}$ is of infinite multiplicity. 
\end{theo}

Recently, in \cite{ MNS23}, J. Manzur, W. Noor and C.F. Santos considered the following discrete semigroup $(W_n)_{n\geq 1}$ on the Hardy space $H^2$, defined by
\[W_nf(z)=(1+z+\cdots + z^{n-1})f(z^{n}).\] They proved the following result. 
\begin{theo}
	For each $n\geq 2$ the operator $W_n^*\in {\mathcal L}(H^2)$  is universal and therefore a positive solution to the invariant subspace probleme on complex Hilbert space is equivalent to proving that each maximal closed invariant subspace of $W_2$ is of codimension $1$.     
\end{theo}    
\begin{proof}
	An easy calculation shows that
	\[W_n^* f(z)=\sum_{k\geq 0}B_n(k)z^k,\]
	where $(a_k)_{k\geq 0}$ are the Taylor coefficients of $f$ and $B_n(k)=a_{nk} + a_{nk+1}+\cdots +a_{nk+n-1}$. 
	Since $W_n^* W_n=(n+1)\Id$, $W_n^*$ is surjective. Moreover $W_n^* f_k=0$ where 
	\[f_k(z)=z^{nk}+ z^{nk+1}+\cdots + z^{nk+n-2} -(n-1)z^{nk+n-1}.\]
	Therefore the kernel of $W_n^*$ is infinite dimensional.
	The universality of $W_n^*$ follows from Theorem~\ref{th:caradus}. 
\end{proof}

\subsection{A model  for semigroups of operators}
In \cite{CCP19}, B. C\'elari\`es, I. Chalendar and J.R. Partington investigated the concept of universal semigroups. Let us give some basic facts on semigroups in order to make this notion more understandable.

 A family $(T_t)_{t\geq 0}$ in  $\LH$ is called a \emph{$C_0$-semigroup} if 
\begin{center}
	$T_0=\Id$, $T_{t+s}=T_tT_s$ for all $s,t\geq 0$ and $\forall x\in\GH$, $\lim_{t\to 0}T_t x=x$. 
\end{center}
A \emph{uniformly continuous semigroup} is a $C_0$-semigroup such that 
\[\lim_{t\to 0}\|T_t-\Id\|=0.\]
Recall also that the generator of a $C_0$-semigroup denoted by $A$ is defined by 
\[Ax=\lim_{t\to 0}\frac{T_tx-x}{t}\]
on $D(A):=\{x:\lim_{t\to 0}\frac{T_tx-x}{t}\mbox{ exists}\}$. Moreover $(T_t)_{t\geq 0}$ is uniformly continuous if and only if $D(A)=\GH$; that is, if and only if $A\in\LH$.  See for example \cite{EN} for an introduction to $C_0$-semigroups.

Since a $C_0$-semigroup $(T_t)_{t\geq 0}$ is not always uniformly continuous, its generator $A$ is 
in general  an unbounded  operator. 
Nevertheless, provided that $1$ is not in the spectrum of $A$, the (negative) Cayley transform of $A$ defined by $V:=(A+\Id)(A-\Id)^{-1}$ is a bounded operator and is called the \emph{cogenerator} of $(T_t)_{t\geq 0}$.    
In \cite[Thm. III.8.1]{nagy-foias} the following equivalence is proved:
\begin{center}
	$V\in \LH$ is the cogenerator of a $C_0$-semigroup of contractions if and only if $V$ is a contraction and $1$ is not an eigenvalue of $V$. 
\end{center}
Not only contractivity is preserved by the cogenerator. Indeed, Sz.-Nagy and Foias \cite[Prop. 8.2]{nagy-foias} proved that a $C_0$-semigroup of contractions consists of normal,  self-adjoint, or unitary operators, if and only if its cogenerator is normal, self-adjoint, or unitary, respectively.     

\subsubsection{Definitions of universality for semigroups}
Let $(S_t)_\tpos$ be the $C_0$-semigroup on $L^2([0,+\infty))$ such that for all $\tpos$, \[S_t: \left\lbrace \begin{array}{rll}
	L^2([0;+ \infty)) & \rightarrow & L^2([0+ \infty)) \\ 
	f & \mapsto & f(\cdot + t)
\end{array}  \right. .\]  

Then for any $t>0$,  by Caradus's theorem, $S_t$ is universal.\\ 

Therefore, for any $C_0$-semigroup $(T_t)_\tpos$ on $L^2([0,+\infty))$, there exist some sequences  $({\mathcal M}_t)_t$ of closed subspaces of $L^2([0,+\infty))$, $(\lambda_t)_t$ of complex numbers and $(R_t)_t$ of bounded isomorphisms from ${\mathcal M}_t$ onto $L^2([0,\infty))$ such that, for every $t>0$, 
\[  T_t=\lambda_t R_t(S_t)_{|{\mathcal M}_t}R_t^{-1}.  \]

This possible definition of universal semigroups is not fully satisfactory since $\lambda_t$, $\mathcal{M}_t$, and $R_t$ depend  heavily on $t$. 
A  more natural and appropriate definition is the following. 

\begin{defi}
	Let $(U_t)_\tpos$ be a $C_0$-semigroup (resp. uniformly continuous) on a Hilbert space $\GH$. It is called a \emph{universal} $C_0$-semigroup (resp. uniformly continuous) if for every $C_0$-semigroup $(T_t)_\tpos$ (resp. uniformly continuous), there exist a closed subspace $\mathcal M$ invariant by every $(U_t)_\tpos$, $\lambda \in \RR$, $\mu \in \RR^{+*}$, and $R:{\mathcal M}\to  \GH$ a bounded isomorphism such that, for all $\tpos$: 
	\[ T_t  =R(e^{\lambda t}U_{\mu t})_{|{\mathcal M}}R^{-1}.   \]  	
\end{defi}
Using this definition of universality for semigroups,  
a certain amount of caution is required: 
for the backward shift semigroup on $L^2(0,\infty)$ each $S_t$ is universal, but the semigroup as a whole is
not, as we shall see later.

It is very natural to find  a criterion involving the generator which captures all the information pertaining to the semigroups. The easiest case to deal with is when the semigroup is uniformly continuous, since then its generator is bounded. In this
situation the semigroup extends to a 
group parametrised by $\RR$. The following theorem is proved in \cite{CCP19}.  
\begin{theo}
	Let $(U_t)_{t \in \RR}$ be a uniformly continuous  group whose (bounded) generator is denoted by  $A$. The following assertions are equivalent:
	\begin{enumerate}
		\item[(i)] for every  uniformly continuous group $(T_t)_{t \in \RR}$, there exist a closed subspace $\mathcal M$ invariant for $(U_t)_{t \in \RR} $, $\mu \ge 0$, and $R:{\mathcal M}\to  \GH$ a bounded isomorphism such that, for all $t \in \RR$: 
		\[ T_t  =R U_{\mu t}{}_{|\mathcal M}R^{-1}.   \] 
		\item[(ii)] $A$ is universal.
	\end{enumerate}
\end{theo}

Here is an example illustrating the previous theorem. 
	Take $A=S_1$. To calculate the group $(U_t)_{t \in \RR}$ it is convenient to
	work with the Fourier transform $\GF$, which, by the Paley--Wiener theorem \cite{rudin87}
	provides an isometric isomorphism between
	$L^2(0,\infty)$ and the Hardy space $H^2(\CC^+)$ of the upper half-plane $\CC^+$.
	Then $S_1^*$ is the right shift by 1 on $L^2(0,\infty)$, and the 
	operator $\GF S_1^* \GF^{-1}$ is the analytic Toeplitz operator with symbol $z \mapsto e^{iz}$.
	
	That is, for $t \in \RR$, $\GF U^*_t  \GF^{-1}$ is the analytic Toeplitz operator with symbol $x\mapsto \exp(t e^{ix})$, where $x \in \RR$,
	and $\GF U_t  \GF^{-1}$ is the anti-analytic Toeplitz operator with symbol
	$x \mapsto \exp(t e^{-ix})$.

Note that the shift semigroup $(S_t)_{t \ge 0}$ on $L^2(0,\infty)$ is not universal even for the class
of all uniformly continuous contraction semigroups.
Its infinitesimal generator $A$ is  defined by $Af=f'$ and hence $\ker (A-\lambda \Id)$
has dimension at most 1 for every $\lambda \in \CC$. Hence if $B$ is a non-zero bounded
operator with kernel of dimension  at least 2, then we cannot have an identity of the form
$ B-\lambda \Id=\mu R(A_{|{\mathcal M}})R^{-1}$, 
and so we have no identity of the
form 
$e^{tB} = e^{\lambda t}R(S_{\mu t})_{|{\mathcal M}}R^{-1}$.

In \cite{CCP19} the authors study examples of universal $C_0$-semigroups of contractions and quasicontractions,
and produce a large class of universal semigroups arising from Toeplitz operators with anti-analytic symbol. 
In addition, they deal with $C_0$-semigroups which are not quasicontractive. 
Under the conditions of concavity and analyticity, which imply the existence of a Wold-type decomposition, they can provide models for such semigroups.

\section{Specific operator classes}
\label{sec:4}

Significant work has focused on restricted classes of operators, yielding original insights:\\

\subsection{Hyponormal Operators} 

A Hilbert space operator $T$ is {\em hyponormal\/} if $T^*T-TT^* \ge 0$. For example, it is well known
that every subnormal operator
(one that is the restriction of a normal operator to a closed invariant subspace) is hyponormal.
Although Scott Brown \cite{Brown78} showed many years ago that every subnormal operator has an invariant
subspace, the question is still open for hyponormal operators.
A recent paper \cite{KL24} proposes various strategies to show that hyponormal operators   have invariant subspaces; these include the study of ``Diagonal + Compact'' operators and the iterates of the Aluthge transform.
We recall that the latter, introduced in 1990 \cite{aluthge}, can be defined as $\widetilde T=|T|^{1/2}U|T|^{1/2}$, where $T$ has
a canonical polar decomposition $T=U|T|$ with $U$ a partial isometry.
These notions may stimulate further research. \\

\subsection{Polynomially Bounded Operators} 
Recall that an operator $T$ on a Banach space $X$  is polynomially bounded if there is a constant $c>0$ such that
$\|p(T)\| \le c  \sup \{|p(\lambda)|: |\lambda|=1 \}$ for all polynomials $p$.
Ambrozie and Müller \cite{AM04} used dual algebra techniques to prove that if $T$ is a polynomially bounded
operator  whose spectrum contains the unit circle then $T^*$
has a nontrivial invariant subspace. Thus, in a reflexive space, $T$ also has an invariant subspace.
Further, if $T$ also satisfies the vanishing condition $\|T^n x\| \to 0$ for all $x \in X$ then $T$ again has
an invariant subspace.

Advances in dual algebra techniques have extended results on invariant subspaces for polynomially bounded operators on Banach spaces, as noted in the 2021 Fields Institute overview \cite{AM21}. 
One recent result is that of Liu~\cite{liu}, who shows that a polynomially bounded operator whose spectrum contains the unit circle
and which satisfies $\|T^*{}^n x^*\| \to 0$ for all $x^* \in X^*$ has itself an invariant subspace.
Further results in this direction are due to R\'ejasse \cite{rejasse}.\\

\subsection{Quasisimilarity} 
It is still unknown whether every non-scalar Hilbert space operator $T$ possesses a 
non-trivial hyperinvariant subspace
(that is, one invariant under all operators commuting with $T$).
Two   operators $A$ and $B$ on a Hilbert space $H$ are said to be quasisimilar if there exist bounded {\em quasiaffinities\/}
(injective mappings with dense range) $X$ and $Y$ such that
$XA = BX$ and $AY = YB$. 
The result that if one of $A$ and $B$ has a hyperinvariant subspace
then so does the other goes back to Hoover (see, e.g. \cite{hoover71,hoover72}).
A generalization of this, {\em ampliation quasisimilarity}, was introduced by Foias and Pearcy \cite{FP05} and similarly
preserves the existence of hyperinvariant subspaces.

A further generalization is pluquasisimilarity, introduced by
Bercovici et al. \cite{BJKP19}. This is the property that there exist families $\GF$ and $\GG$
such that every $X \in \GF$ satisfies $XA=BX$ and every $Y \in \GG$ satisfies $AY=YB$, and such that
$\vee \{\ran X: X \in \GF\}=H$ and
$\cap \{\ker X: X \in \GF\}=\{0\}$, with similar conditions on operators in $\GG$. 
Then we have the following result.

\begin{theo}
Suppose that $A$ and $B$ are pluquasisimilar operators on a Hilbert space and that one of $A$ and $B$ 
has a hyperinvariant subspace.
Then so does the other
\end{theo}

Further work in this direction can be found in \cite{GK23}. It seems that one possible way of
solving the hyperinvariant subspace problem positively could be by showing that every operator
is pluquasisimilar to an operator   known to have hyperinvariant subspaces.

\section{Banach space results}
\label{sec:5}

Although it is known that operators do not always possess invariant subspaces, even
when defined on Banach spaces as friendly as $\ell^1$ and $c_0$ (see \cite{read85} and \cite{read89}),
there have been further developments, some of which we now record.\\

\subsection{Existence and non-existence of invariant subspaces}

In \cite{GR14}, S. Grivaux and M. Roginskaya present a general method for constructing operators without non-trivial invariant closed subsets on a large class of non-reflexive Banach spaces. In particular, their approach clarifies, unifies and generalizes several constructions due to C. Read of operators without non-trivial invariant subspaces on the spaces $\ell^1$, $c_0$ or an $\ell^2$
direct sum of copies of $J$, where $J$ denotes the James space, and without non-trivial invariant subsets on $\ell^1$.  In addition, they explore which geometric properties of a Banach space will ensure that it supports an operator without non-trivial invariant closed subset.   \\


The scalar-plus-compact problem asks whether there exists an infinite dimensional Banach space $X$ such that every bounded linear operator $T\in\LX$ has the form $T=\lambda \Id+K$ with a compact operator $K$. Informally, one says that such a Banach space has ``very few'' operators. This long-standing difficult problem was solved by Argyros and Haydon in \cite{AH11}. Since every compact operator has many invariant subspaces, any operator $T$ on this space $X$ has also many nontrivial  invariant subspaces.\\

Previously, the strongest result in this direction was the theorem by W.T. Gowers and B. Maurey \cite{GM93}  who constructed a Banach space $X$ on which  every bounded linear operator $T\in \LX$ has the form $T=\lambda \Id+S$  with a strictly singular operator $S$.  
A strictly singular operator on a Banach space  is one that is not an isomorphic embedding when restricted to any
infinite-dimensional subspace.
In general the set of strictly singular operators includes the compact operators and the Banach 
space studied in \cite{GM93} does not necessarily provide a  Banach space on which all operators have nontrivial invariant subspaces since C. Read constructed a strictly singular operator on a Banach space induced by a direct sum of James spaces with no nontrivial invariant subspace \cite{read99}.  \\

In \cite{read86}, C. Read conjectured the existence of an operator $T$ on $\ell^1$  such that, for every non-constant polynomial $p$, the operator $p(T)$ has no non-trivial invariant subspace.  In \cite{GR19}, E.A. Gallardo-Guti\'errez and C. Read  constructed such an operator $T$, with the additional properties that its spectrum is reduced to $\{0\}$ and such that, for every non-constant analytic  germ $p$ at $0$, the operator $p(T)$ has no non-trivial invariant subspace. The proof relies on a subtle refinement of the classical Read construction of operators without invariant subspaces.\\

In \cite{CE24} the authors investigate  the existence of non-trivial closed invariant subspaces of
an operator $T$ on Banach spaces whenever its square $T^2$ has or, more generally, whether there exists a polynomial $p$ of degree at least $2$ such that the lattice of invariant subspaces of $p(T)$ is non-trivial. In the Hilbert space setting, the $T^2$-problem was posed by Halmos in the seventies and in 2007, Foias, Jung, Ko and Pearcy~\cite{FJKP} conjectured that it could be equivalent to the Invariant Subspace Problem.
M. Contino and E.A. Gallardo-Guti\'errez~\cite{CE24} gave the following partial answer to Halmos's problem.
\begin{theo}
Let $X$ be an infinite dimensional separable complex Banach space, $T\in\LX$ whose spectrum is reduced to $\{0\}$ and $n\geq1$. Then the following are equivalent:
\begin{enumerate}
	\item $T$ has a nontrivial closed invariant subspace;
	\item $(T-\lambda \Id)^n$ has a nontrivial closed invariant subspace for every $\lambda\in\CC$;
	\item $(T-\lambda \Id)^n$ has a nontrivial closed invariant subspace for some $\lambda\in\CC\setminus\{0\}$.
\end{enumerate} 	
\end{theo}	

\subsection{Typical operators}

We now describe some recent work which, in a sense, tells us what 
properties are possessed by a ``typical'' operator.

%
%
%
%
%
%
%
%
%
%

 \begin{defi} 
 	Let $(M,\tau)$ be a topological space. A property $(P)$ of elements of $M$ is typical
 	for the topology $\tau$ if the set
 	\[  \{x\in M: x\mbox{ has }(P)\}\]
 	is comeager in $(M,\tau)$; i.e., it contains an intersection of countable dense open sets. 
 	
 	In the case where a property $(P)$ of elements of a topological space $(M,\tau)$ is typical, we say that
 	``a typical $x \in (M , \tau )$ has $(P)$''.
 \end{defi}
 Let us now consider typical properties of contractions. Given a Banach space $X$, we denote by $\BX$ the set of contractions on $X$, that is,
 \[\BX:=\{T \in\LX :\|T\|\leq 1\}.\]
 A study of  typical properties was initiated by T. Eisner for contractions for different topologies on $\BX$, namely 
 \begin{itemize}
 \item the operator norm topology,
 \item the   Strong Operator Topology, denoted by SOT, and defined by  the pointwise convergence topology on $\LX$,
 \item the Strong* Operator Topology, denoted by  SOT*, defined by  the pointwise convergence topology for operators and their adjoints,
 \item the Weak Operator Topology, denoted by WOT, defined  by the weak pointwise convergence topology on $\LX$. 
\end{itemize}

These have the following properties:
\begin{itemize}
\item If $X$ is separable, then $(\BX,SOT)$ is a Polish space (separable, and completely metrizable).

\item If $X^*$ is separable, then $(\BX,SOT)$ and $(\BX,SOT^*)$ are Polish.

\item If $X$ is reflexive and separable, then $(\BX,WOT)$ is Polish.\\
\end{itemize}
As a corollary it follows from Baire's theorem that for $X=\ell^2$ a set being comeager implies its density for the WOT, SOT and SOT* topology;
 for $X=\ell^1$, being comeager with respect to the SOT topology also implies   density.
  
 In \cite{eisner2010} T. Eisner showed that, with respect to the Weak Operator Topology, a typical contraction on a separable Hilbert space is unitary. 
  This concept was further developed by Eisner and M\'atrai in \cite{em13}, who studied typical properties of contractions on separable Hilbert spaces for other topologies, such as the Strong Operator Topology.  For $X=\ell^2$, they proved that a typical contraction 
 $T \in  (\BX,SOT)$ is unitarily equivalent to the infinite-dimensional
 unilateral backward shift on 
 \[\ell^2(\mathbb N,\ell^2) := \left\{(u_n)_{n\geq 1} : u_n\in  \ell^2 \mbox{ for every n} \geq  1 \mbox{ and }\sum_{n\geq 1}\|u_n\|_{\ell^2}<\infty \right\}\]
 defined by $B_\infty(u_1,u_2,u_3,...) = (u_2,u_3,...)$. These results imply in particular that, for $X=\ell^2$,  
 a typical contraction $T   \in  (\BX,SOT)$ or $T   \in  (\BX,WOT)$ has many invariant subspaces. 
 
 A few years later, Grivaux, Matheron and Menet \cite{GMM21}  investigated typical properties of contractions on $\ell^p$-spaces.  Their main goal was to determine whether, for $X=\ell^p,$ a typical contraction $T\in (\BX,SOT)$ or $T\in (\BX,SOT^*)$  has a nontrivial invariant subspace.   They proved that a typical contraction, for $X=\ell^2$,  $T \in  (\BX,SOT^*)$ has a nontrivial invariant subspace, by using the fact that a typical contraction for the $SOT^*$ topology  has a  spectrum containing the unit circle.
 
 They also proved that, for $X=\ell^1$ and the Strong Operator Topology,  a typical contraction $T$   has a nontrivial invariant subspace. We know, however, that there are operators on $\ell^1$ with no invariant subspaces (see \cite{read85}).
 
 Very recently,  typical properties of positive contractions were investigated by V. Gillet \cite{gillet1,gillet2}.

\end{document}